\documentclass[12pt, reqno]{amsart}
\usepackage{amsmath, amsthm, amscd, amsfonts, amssymb, graphicx, color}

\newtheorem{df}{Definition}[section]
\newtheorem{thm}[df]{Theorem}

\newtheorem{rema}[df] {Remark}

\begin{document}
\setcounter{page}{1}

\title[On Cartan Joint Spectra]{On Cartan Joint Spectra}

\author{Enrico Boasso}

\begin{abstract} In this work several results regarding the Cartan version
of the Taylor, the S\l odkowski, the Fredholm, the split and the
Fredholm split joint spectra will be studied.\par
\vskip.2cm\noindent KEYWORDS: Joint spectra, representations of
solvable Lie algebras.\par \vskip.2cm \noindent MSC (2000): Primary
47A13; Secondary 47A53, 47A80, 46M05, 17B30, 17B55.
\end{abstract}

\maketitle
\section{ Introduction}

\indent The main concern of the present work is noncommutative
spectral theory within the Lie algebra framework. During the last
years several joint spectra have been extended from commuting tuples
of operators to representations of complex solvable finite
dimensional Lie algebras in Banach spaces. Some of the spectra that
have been considered are the Taylor, the S\l odkowski, the Fredholm,
the split and the Fredholm split joint spectra, see for example the
works [12], [5], [2], [16], [17] and [3]. As regard these spectra in
the commuting case, see for instance [19], [18], [11], [14] and [9].
\par

\indent The most important difference between the nilpotent and the
solvable Lie algebra setting consists in the so-called projection
property. In fact, this property holds only for ideals in the
solvable Lie algebra case, while it also holds for subalgebras when
nilpotent Lie algebras are considered, see [12], [5], [2] and [17].
The problem of introducing a well-behaved joint spectrum in the
solvable Lie algebra framework which extends the Taylor joint
spectrum and at the same time preserves the projection property for
subalgebras was solved in [1]. Moreover, this new version of the
Taylor joint spectrum coincides with the one in [12], [5] and [16]
for nilpotent Lie algebras, but it differs in general from the
spectrum in these articles for solvable non nilpotent Lie
subalgebras of operators, see [1] and [5]. Since in order to define
the joint spectrum introduced in [1] Cartan subalgebras were
considered, it was suggested that this spectrum should be called the
Cartan-Taylor joint spectrum.\par

\indent Furthermore, in [8] the Cartan-S\l odkowski joint spectra of
a representation of a complex solvable finite dimensional Lie
algebra in a Banach space were introduced. The joint spectra of [8]
consist in the Cartan version of the S\l odkowski joint spectra and
they provide a new extension of the spectra of [18] from commutative
tuples of operators to representations of solvable Lie algebras in
Banach spaces. In fact, as for the spectrum of [1], the joint
spectra under consideration coincide with the ones of [2] and [17,
2.11] in the nilpotent Lie algebra case, but they differ in general
from these spectra in the solvable non nilpotent Lie algebra
setting, see [5] and [8]. Furthermore, as for the Cartan-Taylor
joint spectrum, the Cartan-S\l odkowski joint spectra have  the
projection property for subalgebras even in the solvable
setting.\par

\indent The main concern of the present work consists in the study
of several spectra related to the Cartan-Taylor and the Cartan-S\l
odkowski joint spectra. In fact, in section 2 the Cartan version of
the Fredholm, the split and the Fredholm split joint spectra will be
introduced and their main spectral properties will be proved. In
addition, the relationships among these joint spectra will be
considered. Furthermore, in section 3 all the aforesaid joint
spectra of a solvable Lie algebra representation by compact
operators in a Banach space will be described. Moreover, the
behavior of all the above-mentioned joint spectra with respect to
the procedure of passing from two representations of complex
solvable finite dimensional Lie algebras in Banach spaces to the
tensor product and the multiplication representation of the direct
sum of the algebras will be studied. All these results will provide
the Cartan version of several properties considered in [3] and
[4].\par

\indent The results of [3] were communicated in a talk given during
the 20th International Conference on Operator Theory held in
Timisoara, Rumania, from 30th June to 5th July 2004. It is a
pleasure for the author to acknowledge the stimulating atmosphere
and the warm hospitality of the Conference. Moreover, the author
wish to express his fully indebtedness to the organizers of the
Conference, specially to Luminita Stafi, Dan Timotin and
Florian-Horia Vasilescu.\par

 \section{Spectra related to the Cartan-Taylor joint spectrum}

\indent In this section the Cartan version of several joint spectra
will be introduced. First of all, some notation is needed.
\par

\indent From now on, if $X$ and $Y$ are two Banach spaces, then
$L(X,Y)$ denotes the algebra of all linear and continuous operators
defined on $X$ with values in $Y$, and $K(X,Y)$ the closed ideal of
all compact operators of $L(X,Y)$. As usual, when $X=Y$, $L(X,X)$
and $K(X,X)$ are denoted by $L(X)$ and $K(X)$ respectively. \par

\indent Consider a Banach space $X$, a complex solvable finite
dimensional Lie algebra $L$, and a representation $\rho\colon L\to
L(X)$. As regard the definitions and the main properties of the
Taylor, the S\l odkowski, the Fredholm, the split and the Fredholm
split joint spectra of the representation $\rho$, see [12], [5],
[2], [16], [17] and [3]. In the commutative framework, see [19],
[18], [11], [14] and  [9].\par

\indent The following remark will be central in the definition of
the Cartan spectra.\par

\begin{rema} \rm Let $L$ be a complex solvable
finite dimensional Lie algebra and $H$ a nilpotent subalgebra of
$L$. Consider the representation
$$
ad_L\colon H\colon \to End(L),\hskip1cm ad_L(h)(l)=[h,l],
$$
where $h$ and $l$ belong to $H$ and $L$ respectively. Then,
according to Theorem 1 of [6], there are finite linear functionals
defined on $H$, $\alpha_i$, $i=1,\ldots ,m$, such that
$$
L=\bigoplus_{i=1,\ldots ,m} L^{\alpha_i},
$$
where $L^{\alpha_i}=\{ l\in L\colon \hbox{ there is }q\in\Bbb
N\hbox{ such that } (ad_L(h)-\alpha_i(h))^q$ $(l)=0,\forall\hbox{ }
h\in H\}$. Furthermore, $L^{\alpha_i}$ are vector subspaces of $L$
such that $ad_L(L^{\alpha_i})\subseteq L^{\alpha_i}$, and
$L^{\alpha_i}\neq 0$, $i=1,\ldots ,m$. \par \indent Now well, since
$H$ is a nilpotent subalgebra of $L$, $H\subseteq L^0$, where $L^0$
denotes the subspaces defined by the null lineal functional of $H$.
A nilpotent subalgebra is said a \it{Cartan subalgebra} \rm of $L$
if $H=L^0$. If $L$ is a complex solvable finite dimensional Lie
algebra, then $L$ has at least one Cartan subalgebra, see Theorem 2
of [6] or Theorem 1, Chapter III, of [13].\par \indent As in [1] and
in [8], given $H$ a Cartan subalgebra of $L$, consider the Cartan
decomposition of $L$ determined by $H$, that is
$$
L= H\oplus H_*,
$$
where $H_*$ is the direct sum of all the subspaces $L^{\alpha_i}$,
$i\in [\![1,m]\!]$, such that $\alpha_i\neq 0$.\par \indent On the
other hand, note that if $L_1$ and $L_2$ are two complex solvable
finite dimensional Lie algebras, and if $H_1$ and $H_2$ are two
Cartan subalgebras of $L_1$ and $L_2$ respectively, then a
straightforward calculations shows that $H=H_1\oplus H_2$ is a
Cartan subalgebra of $L=L_1\times L_2$, the direct sum of $L_1$ and
$L_2$. Furthermore, $L$ can be decomposed as
$$
L= H\oplus H_*= (H_1\oplus H_2)\oplus (H_{1*}\oplus H_{2*}),
$$
where $H_*=H_{1*}\oplus H_{2*}$ and
$$
L_1=H_1\oplus H_{1*},\hskip1cm L_2=H_1\oplus H_{2*}.
$$
\end{rema}
\indent Next the definitions of the Cartan-Taylor and the Cartan-S\l
odkowski joint spectra will be reviewed, see [1] and [8]. Recall
that if $\rho\colon L\to L(X)$ is a representation of the Lie
algebra $L$ in the Banach space $X$, then $\sigma (\rho)$,
$\sigma_{\delta ,k}(\rho)$ and $\sigma_{\pi ,k}(\rho)$ denote
respectively the Taylor and the S\l odkowski joint spectra of
$\rho$, see [12], [5], [2], [16] and [17].\par

\begin{df} Let $X$ be a complex Banach space, $L$ a
complex solvable finite dimensional Lie algebra, $\rho\colon L\to
L(X)$ a representation of $L$ in $X$, and $H$ a Cartan
subalgebra of $L$. Then, the Cartan-Taylor joint spectrum of $\rho$
is the set
$$
\Sigma(\rho)= \{f\in L^*: \hbox{
}f\mid_H\in\sigma(\rho\mid_H),\hbox{ and f vanishes on }  H_*\}.
$$
\indent In addition, the $k$-th $\delta$-Cartan-S\l odkowski joint
spectrum of $\rho$ is the set
$$
\Sigma_{\delta ,k}(\rho)=  \{f\in L^*: \hbox{
}f\mid_H\in\sigma_{\delta ,k}(\rho\mid_H),\hbox{ and f vanishes on }
H_*\},
$$
and the $k$-th $\pi$-Cartan-S\l odkowski joint spectrum of $\rho$ is
the set
$$
\Sigma_{\pi ,k}(\rho)= \{f\in L^*: \hbox{  }f\mid_H\in\sigma_{\pi
,k}(\rho\mid_H),\hbox{ and f vanishes on }  H_*\},
$$
where $\rho\mid_H\colon H\to L(X)$ is the restriction of $\rho$ to
the subalgebra $H$, $L=H\oplus H_*$ is the Cartan decomposition of
$L$, and $k=0,\ldots , n=\dim L$. \par \indent Observe that
$$
\Sigma_{\delta ,n}(\rho)= \Sigma_{\pi ,n} (\rho)= \Sigma(\rho).
$$
\end{df}

\begin{rema}\rm First of all, note that according
to [17, 2.4.4] and to [17, 0.5.8], the definition of the Cartan-S\l
odkowski joint spectra given above coincides with the one introduced
in section 2 of [8].\par \indent Next, according to Lemma 2 of [8],
the definition of each Cartan-S\l odkowski joint spectrum is
independent of the particular Cartan subalgebra of $L$ used to
define it. In addition, according to Lemma 1 of [8], each Cartan-S\l
odkowski joint spectrum is contained in the sets of characters of
$L$. As regard the Cartan-Taylor joint spectrum, see Theorems 2.1
and 2.6 of [1].\par \indent Furthermore, according to Theorem 1 of
[8], each Cartan-S\l odkowski joint spectrum has the projection
property for Lie subalgebras of $L$, that is, if $E$ is a subalgebra
of $L$ and if $\pi\colon L^*\to E^*$ denotes the restriction map,
then
$$
\pi(\Sigma_*(\rho))=\Sigma_*(\rho\mid_E),
$$
where $\Sigma_*$ denotes one of the joint spectra in Definition 2.2
and $\rho\mid_E\colon E\to L(X)$ is the restriction of the
representation $\rho$ to the subalgebra $E$. As regard the
Cartan-Taylor spectrum, see Theorem 2.6 of [1].\par \indent What is
more, since the S\l odkowski joint spectra are compact subsets of
the set of characters of $L$, see [24, 2.11.3] or Theorems 4 and 7
of [2], a straightforward calculation shows that each Cartan-S\l
odkowski joint spectrum is a compact subset of the characters of
$L$.\par \indent On the other hand, according to Theorem 2.6 of [1]
and to the example that follows Definition 1 of [5], if $L$ is a
solvable not nilpotent Lie algebra, then $\Sigma_{\delta ,k}(\rho)$
(resp. $\Sigma_{\pi ,k}(\rho))$ differs in general from
$\sigma_{\delta ,k}(\rho)$ (resp. $\sigma_{\pi ,k}(\rho)$),
$k=0,\ldots ,n$.\par \indent However, when $L$ is a nilpotent Lie
algebra
$$
\Sigma_{\delta ,k}(\rho)=\sigma_{\delta ,k}(\rho),\hskip1cm
\Sigma_{\pi ,k}(\rho)=\sigma_{\pi ,k}(\rho),
$$
$k=0,\ldots ,n$. In fact, in this case, since $H=L$ is a Cartan
subalgebra of $L$, see [6] or Chapter 3 of [13], $H_*=0$, and
consequently the equalities are true.\par
\end{rema}
\indent Next the Cartan version of the Fredholm, the split and the
Fredholm split joint spectra of a solvable Lie algebra
representation in a Banach space will be introduced. Recall that if
$\rho\colon L\to L(X)$ is a representation of the Lie algebra $L$ in
the Banach space $X$, then $\sigma_e (\rho)$, $\sigma_{\delta
,k,e}(\rho)$ and $\sigma_{\pi ,k,e}(\rho)$ denote the Fredholm joint
spectra of $\rho$, see [3], $sp(\rho)$, $sp_{\delta , k}(\rho)$ and
$sp_{\pi, k}(\rho)$ denote the split joint spectra of $\rho$, see
[17] and [3], and $sp_e(\rho)$, $sp_{\delta , k, e}(\rho)$ and
$sp_{\pi, k, e}(\rho)$ denote the Fredholm split joint spectra of
$\rho$, see [3].\par

\begin{df} Let $X$, $L$, $\rho\colon L\to L(X)$, and $H$ be as in Definition 2.2. Then, the Cartan-Fredholm
or Cartan-essential Taylor joint spectrum of $\rho$ is the set
$$
\Sigma_e(\rho)=\{ f\in L^*\colon f\mid_H\in\sigma_e(\rho\mid_H),
\hbox{ and f vanishes on }  H_*\}.
$$
\indent In addition, the $k$-th Cartan-Fredholm or Cartan-essential
$\delta$-S\l odkowski joint spectra of $\rho$ is the set
$$
\Sigma_{\delta , k ,e}(\rho)=\{ f\in L^*\colon
f\mid_H\in\sigma_{\delta,k,e}(\rho\mid_H), \hbox{ and f vanishes on
}  H_*\},
$$
and the $k$-th Cartan-Fredholm or Cartan-essential $\pi$-S\l
odkowski joint spectrum of $\rho$ is the set
$$
\Sigma_{\pi ,k ,e}(\rho)=\{ f\in L^*\colon
f\mid_H\in\sigma_{\pi,k,e}(\rho\mid_H), \hbox{ and f vanishes on }
H_*\},
$$
where $k=0,\ldots , n=\dim L$, and $\rho\mid_H\colon H\to L(X)$ is
the restriction of the representation $\rho$ to $H$.\par \indent
Observe that $\Sigma_e(\rho)=\Sigma_{\delta ,n ,e}(\rho)=
\Sigma_{\pi ,n ,e}(\rho)$.\par

\end{df}

\begin{df}Let $X$, $L$, $\rho\colon L\to L(X)$, and $H$ be as in Definition 2.2. Then, the Cartan-split
spectrum of $\rho$ is the set
$$
Sp(\rho)=\{ f\in L^*\colon f\mid_H\in sp(\rho\mid_H),\hbox{ and f
vanishes on }  H_*\}.
$$
\indent In addition, the $k$-th Cartan $\delta$-split joint spectrum
of $\rho$ is the set
$$
Sp_{\delta ,k}(\rho)= \{ f\in L^*\colon f\mid_H\in
sp_{\delta,k}(\rho\mid_H),\hbox{ and f vanishes on }  H_*\} ,
$$
and the $k$-th Cartan $\pi$-split joint spectrum of $\rho$ is the
set
$$
Sp_{\pi ,k}(\rho)= \{ f\in L^*\colon f\mid_H\in
sp_{\pi,k}(\rho\mid_H),\hbox{ and f vanishes on }  H_*\},
$$
where $k=0,\ldots , n = \dim L$, and $\rho\mid_H\colon H\to L(X)$ is
the restriction of the representation $\rho$ to $H$.\par \indent
Observe that  $Sp_{\delta ,n}(\rho)= Sp_{\pi ,n}(\rho)=
Sp(\rho)$.\par
\end{df}

\begin{df}Let $X$, $L$, $\rho\colon L\to L(X)$, and $H$ be as in Definition 2.2. Then, the Cartan-Fredholm
or Cartan-essential split spectrum of $\rho$ is the set
$$
Sp_e(\rho)=\{ f\in L^*\colon f\mid_H\in sp_e(\rho\mid_H),\hbox{ and
f vanishes on }  H_*\}.
$$
\indent In addition, the $k$-th Cartan-Fredholm or Cartan-essential
$\delta$-split joint spectrum of $\rho$ is the set
$$
Sp_{\delta ,k,e}(\rho)= \{ f\in L^*\colon f\mid_H\in
sp_{\delta,k,e}(\rho\mid_H), \hbox{ and f vanishes on }  H_*\} ,
$$
and the $k$-th Cartan-Fredholm or Cartan-essential $\pi$-split joint
spectrum of $\rho$ is the set
$$
Sp_{\pi ,k,e}(\rho)= \{ f\in L^*\colon f\mid_H\in
sp_{\pi,k,e}(\rho\mid_H), \hbox{ and f vanishes on }  H_*\},
$$
where $k=0,\ldots , n =\dim L$, and $\rho\mid_H\colon H\to L(X)$ is
the restriction of the representation $\rho$ to $H$.\par \indent
Observe that   $Sp_{\delta ,n,e}(\rho)= Sp_{\pi ,n,e}(\rho)=
Sp_e(\rho)$.\par
\end{df}

\indent In what follows the properties of the sets introduced in
definitions 2.4, 2.5 and 2.6 will be studied. In order to present
the relationships among these sets and the Cartan-Taylor and the
Cartan-S\l odkowski joint spectra, some definitions will be
reviewed. \par

\indent As before, consider a complex solvable finite dimensional
Lie algebra $L$, a complex Banach space $X$, and $\rho\colon L\to
L(X)$ a representation of $L$ in $X$. Then, if $C(X)=L(X)/K(X)$,
$L_{\rho}\colon L\to L(L(X))$ and $\tilde{L}_{\rho}\colon L\to
L(C(X))$ are the representations induced by left multiplication of
operators in $L(X)$ and $C(X)$ respectively, which were studied in
section 3 of [3] and [17].\par

\indent On the other hand, $L^{op}$ will denote the opposite algebra
of $L$, that is the algebra that coincides with $L$ as vector space
and that has the bracket $[x,y]^{op}=-[x,y]=[y,x]$, where $x$ and
$y$ belong to $L^{op}=L$ and $[.,.]$ is the Lie bracket of $L$.
Then, if $X{'}$ is the dual space of $X$, $\rho^*\colon L\to
L(X{'})$ is the adjoint representation of $\rho$, see [2] and [17].
In addition, $R_{\rho}\colon L^{op}\to L(L(X))$ and
$\tilde{R}_{\rho}\colon L^{op}\to L(C(X))$ are the representations
induced by right multiplication of operators in $L(X)$ and $C(X)$
respectively, which were studied in [3] and [17].\par

\indent Next denote by $l^{\infty}(X)$ the Banach space of all
bounded sequences of elements of $X$ with the supremum norm. Let
$m(X)$ be the set of all sequences $(x_n)_{n\in \Bbb N}$ such that
the closure of the set $\{x_n\colon n\in \Bbb N\}$ is compact. Then
$m(X)$  is a closed subspace of $l^{\infty}(X)$. Denote
$\tilde{X}=l^{\infty}(X)/m(X)$.\par

\indent If $X$ and $Y$ are Banach spaces, and if $T\in L(X,Y)$, then
$T$ defines pointwise an operator $T^{\infty}\colon l^{\infty}(X)\to
l^{\infty}(Y)$ by $T^{\infty}((x_n)_{n\in \Bbb N})=(T(x_n))_{n\in
\Bbb N}$. It is clear that $T^{\infty}(m(X))\subseteq m(Y)$. Denote
by $\tilde{T}\colon \tilde{X}\to \tilde{Y}$ the operator induced by
$T^{\infty}$.\par

\indent Now well, consider $X$, $L$ and $\rho\colon L\to L(X)$ as
before, and define the map
$$
\tilde{\rho}\colon L\to L(\tilde{X}),\hskip1cm
\tilde{\rho}(l)=\widetilde{\rho(l)},
$$
where $l\in L$. According to Theorem 5 of [15], $\tilde{\rho}\colon
L\to L(\tilde{X})$ is a representation of $L$ in $\tilde{X}$, see
also [7].\par

\begin{thm} Let $X$ be a complex Banach space, $L$ a
complex solvable finite dimensional Lie algebra, and $\rho\colon
L\to L(X)$ a representation of $L$ in $X$. Then,\par
\hskip.2cm i) $\Sigma_{\pi,k}(\rho)=\Sigma_{\delta,k}(\rho^*)$,
$\Sigma_{\delta,k}(\rho)=\Sigma_{\pi,k}(\rho^*)$,
$\Sigma(\rho)=\Sigma(\rho^*)$,\par \hskip.2cm ii)
$\Sigma_{\pi,k,e}(\rho)=\Sigma_{\delta,k,e}(\rho^*)$,
$\Sigma_{\delta,k,e}(\rho)=\Sigma_{\pi,k,e}(\rho^*)$,
$\Sigma_e(\rho)=\Sigma_e(\rho^*)$,\par \hskip.2cm iii)
$\Sigma_{\delta,k,e}(\rho)=\Sigma_{\delta,k}(\tilde{\rho})$,
$\Sigma_{\pi,k,e}(\rho)=\Sigma_{\pi,k}(\tilde{\rho})$,
$\Sigma_e(\rho)=\Sigma(\tilde{\rho})$,\par \hskip.2cm iv)
$\Sigma_{\pi,k,e}(\rho)= \Sigma_{ \delta,k}((\tilde{\rho})^*)$,
$\Sigma_e(\rho)=\Sigma((\tilde{\rho})^*)$,\par \hskip.2cm v)
$Sp_{\delta,k}(\rho)=\Sigma_{\delta ,k}(L_{\rho})$,
$Sp_{\pi,k}(\rho)=\Sigma_{\delta ,k}(R_{\rho})$,
$Sp(\rho)=\Sigma(L_{\rho})=$\par \hskip.6cm $\Sigma(R_{\rho})$,\par
\hskip.2cm vi) $Sp_{\delta ,k,e}(\rho)=\Sigma_{\delta
,k}(\tilde{L}_{\rho})$, $Sp_{\pi ,k,e}(\rho)=\Sigma_{\delta
,k}(\tilde{R}_{\rho})$, $Sp_e(\rho)=\Sigma(\tilde{L}_{\rho})=$\par
\hskip.6cm $\Sigma(\tilde{R}_{\rho})$,\par\par where $k=0,\ldots
,n=$ dim $L$.
\end{thm}
\begin{proof} It is a consequence of
Definitions 2.2, 2.4, 2.5 and 2.6, Theorem 7 of [2], [17, 2.11.4],
Theorem 4 of [3], [17, 3.1.5], [17, 3.1.7], Theorem 8 of [3],
Proposition 3.1 of [7], and [17, 0.5.8].\end{proof}

\begin{thm} Let $X$ be a complex Banach space, $L$ a
complex solvable finite dimensional Lie algebra, and $\rho\colon
L\to L(X)$ a representation of $L$ in $X$. Then,\par
\hskip.2cm i) the sets $\Sigma_e(\rho)$,
$\Sigma_{\delta,k,e}(\rho)$, $\Sigma_{\pi,k,e}(\rho)$, $Sp(\rho)$,
$Sp_{\delta,k}(\rho)$, $Sp_{\pi,k}(\rho)$, $Sp_e(\rho)$,
$Sp_{\delta,k,e}(\rho)$, and $Sp_{\pi,k,e}(\rho)$ are independent of
the particular Cartan subalgebra of $L$ used to define them,\par
\hskip.2cm ii) all the sets considered in i) are compact nonempty
subsets of the character of $L$ that have the projection property
for Lie subalgebras of $L$.
\end{thm}
\begin{proof} It is a consequence of Theorem
2.7 and Lemma 1, Lemma 2 and Theorem 1 of [8], see also Proposition
2.1 and Theorem 2.6 of [1].\end{proof}

\begin{rema}\rm In the above conditions, it is clear that

\begin{align*} &\Sigma_{\delta ,k}(\rho)\subseteq Sp_{\delta ,k}(\rho),
\hskip.5cm\Sigma_{\pi ,k}(\rho)\subseteq
Sp_{\pi ,k}(\rho), \hskip.5cm\Sigma(\rho)\subseteq Sp(\rho),\\
&\Sigma_{\delta ,k, e}(\rho)\subseteq Sp_{\delta ,k,
e}(\rho),\hskip.1cm \Sigma_{\pi ,k,e}(\rho)\subseteq Sp_{\pi
,k,e}(\rho),\hskip.1cm \Sigma_e(\rho)\subseteq
Sp_e(\rho).\\\end{align*}

Furthermore, if $X$ is a Hilbert space, then the above inclusions
are equalities.\par

\indent In addition, according to Theorem 2.6 of [1] and to the
example that follows Definition 1 of [5], note that if $L$ is a
solvable not nilpotent Lie algebra, then $Sp_{\delta ,k}(\rho)$
(resp. $Sp_{\pi ,k}(\rho))$ differs in general from $sp_{\delta
,k}(\rho)$ (resp. $sp_{\pi ,k}(\rho)$), $k=0,\ldots ,n$.\par \indent
However, when $L$ is a nilpotent Lie algebra\par \hskip.2cm i)
$\Sigma_{\delta ,k,e}(\rho)=\sigma_{\delta ,k,e}(\rho)$,
$\Sigma_{\pi ,k,e}(\rho)=\sigma_{\pi ,k,e}(\rho)$,\par \hskip.2cm
ii) $Sp_{\delta ,k}(\rho)=sp_{\delta ,k}(\rho)$, $Sp_{\pi
,k}(\rho)=sp_{\pi ,k}(\rho)$,\par \hskip.2cm iii) $Sp_{\delta
,k,e}(\rho)=sp_{\delta ,k,e}(\rho)$, $Sp_{\pi ,k,e}(\rho)=sp_{\pi
,k,e}(\rho)$,\par where $k=0,\ldots ,n$. In fact, in this case,
since $H=L$ is a Cartan subalgebra of $L$, see [6] or Chapter 3 of
[13], $H_*=0$, and consequently the equalities are true.\par
\end{rema}

\section{Properties of the joint spectra}

 \indent In
this section several results regarding the joint spectra of
Definitions 2.2  and 2.4 - 2.6 will be studied. In first place,
solvable Lie algebra representations by compact operators in Banach
spaces will be considered both in the infinite and finite
dimensional case.\par

\begin{thm} Let $X$ be an infinite dimensional complex
Banach space, $L$ a complex solvable finite dimensional Lie algebra,
and $\rho\colon L\to L(X)$ a representation of $L$ in $X$, such that
$\rho(l)\in K(X)$ for each $l\in L$. In addition, suppose that $H$
is a Cartan subalgebra of $L$. Then, \par \hskip.2cm i) $\Sigma
(\rho)=Sp(\rho)$, $\Sigma_{\delta ,k}(\rho)=Sp_{\delta ,k}(\rho)$
and $\Sigma_{\pi ,k}(\rho)=Sp_{\pi ,k}(\rho)$,\par \hskip.2cm ii)
the sets $\Sigma (\rho)$, $\Sigma_{\delta ,k}(\rho)\cup \{0\}$,
$\Sigma_{\pi ,k}(\rho)\cup \{0\}$, $Sp(\rho)$, $Sp_{\delta
,k}(\rho)\cup \{0\}$ and $Sp_{\pi ,k}(\rho)\cup \{0\}$ coincide with
the set

\begin{align*} \{0\}\cup\{& f\in L^*: f(L^2)=0, \hbox{  such that there
is  }x\in X,\hbox{  }x\neq 0, \\
&\hbox{with the property  }\rho(h)=f(h)x,\hbox{  } \forall\hbox{  }
h\in H\},\\\end{align*}

\hskip.6cm iii)
$\Sigma_e(\rho)=\Sigma_{\delta,k,e}(\rho)=\Sigma_{\pi,k,e}(\rho)=\{0\}$,\par
\hskip.2cm iv) $sp_{e}(\rho)=sp_{\delta ,k ,e}(\rho)=sp_{\pi
,k,e}(\rho)=\{0\}$,\par where $k=0,\ldots , n=$ dim L.\par
\end{thm}
\begin{proof} It is a consequence of
Definitions 2.2, 2.4, 2.5 and 2.6, and Theorems 1, 2 and 3 of
[4].\end{proof}

In the conditions of Theorem 3.1,
note that when $X$ is finite dimensional, the spectra to be
considered are the Cartan-Taylor and the Cartan-S\l odkowski joint
spectra. \par 

\begin{thm} Let $X$ be a complex finite
dimensional Banach space, $L$ a complex solvable finite dimensional
Lie algebra, and $\rho\colon L\to L(X)$ a representation of $L$ in
$X$. In addition, suppose that $H$ a Cartan subalgebra of $L$.
Then,\par

\begin{align*} &\Sigma(\rho)=\Sigma_{\delta ,k}(\rho)=\Sigma_{\pi, k}(\rho)=
\{f\in L^*: f(L^2)=0, \hbox{ such that there is }\\
& \hskip1.1cmx\in X,\hbox{ } x\neq 0, \hbox { with the property:
}\rho (h)x=f(h)x,\hbox{  }\forall\hbox{  } h\in H\}.\\\end{align*}

 \end{thm}
\begin{proof} It is a consequence of
Definition 2.2 and Theorem 4 of [4].\end{proof}

\indent Next the spectral contributions of two representations of
complex solvable finite dimensional Lie algebras in Banach spaces to
the Cartan-Taylor, the Cartan-S\l odkowski, the Cartan-Fredholm, the
Cartan-split and the Cartan-Fredholm split joint spectra of the
tensor product and the multiplication representations of the direct
sum of the algebras will be described. The results to be presented
consist in an extension to these spectra of the ones obtained in
[3].\par

\indent In what follows the axiomatic tensor product introduced in
[10] will be considered. As regard its definition as well as the
definition of the class of ideals between Banach spaces to be used,
see [10] and [3]. Furthermore, in order to review the definition of
the tensor product and the multiplication representation, see [10],
[17] and [3]. Recall that if $X$ is a Banach space, then $X{'}$ will
denote the dual space of $X$.\par

\begin{thm}  Let $X_1$ and $X_2$ be two complex Banach
spaces, $L_1$ and $L_2$ two complex solvable finite dimensional Lie
algebras, and $\rho_i \colon L_i\to  L(X_i)$, $i=1$, $2$,
two representations of the Lie algebras. Suppose that
$X_1\tilde{\otimes} X_2$ is a tensor product of $X_1$ and $X_2$
relative to $ \langle X_1,X_1{'}\rangle$ and $\langle
X_2,X_2{'}\rangle$, and consider the tensor product representation
of $L=L_1\times L_2$, $\rho\colon L\to L(X_1\tilde{\otimes} X_2)$.
Then,

\begin{align*} i) \hbox{}&\bigcup_{p+q=k}\Sigma_{\delta ,p}(\rho_1)\times
\Sigma_{\delta ,q}(\rho_2)\subseteq\Sigma_{\delta ,k}(\rho)
\subseteq Sp_{\delta ,k } (\rho)\subseteq \\
&\bigcup_{p+q=k}Sp_{\delta ,p} (\rho_1)\times Sp_{\delta,q}(\rho_2),\\\end{align*}

\begin{align*} ii)\hbox{}&\bigcup_{p+q=k}\Sigma_{\pi ,p}(\rho_1)\times
\Sigma_{\pi ,q}(\rho_2)\subseteq\Sigma_{\pi ,k}(\rho)
\subseteq Sp_{\pi ,k} (\rho)\subseteq\\
&\bigcup_{p+q=k}Sp_{\pi ,p} (\rho_1)\times Sp_{\pi
,q}(\rho_2).\\\end{align*}
\indent In particular, if $X_1$ and $X_2$ are Hilbert spaces, the
above inclusions are equalities.
\end{thm}
\begin{proof}It is a consequence of
Definitions 2.2 and  2.5, Theorem 14 of [3] and Remark 2.1.
\end{proof}

\begin{thm} In the same conditions of Theorem 3.3,

\begin{align*} i) \hbox{}&\bigcup_{p+q=k}\Sigma_{\delta ,p,e}(\rho_1)\times
\Sigma_{\delta ,q}(\rho_2 )\bigcup\bigcup_{p+q=k}\Sigma_{\delta
,p}(\rho_1) \times\Sigma_{\delta ,q ,e}(\rho_2)
\subseteq\sigma_{\delta ,k ,e}(\rho)
\subseteq \\
&Sp_{\delta ,k,e} (\rho)\subseteq \bigcup_{p+q=k}Sp_{\delta ,p,e}
(\rho_1)\times Sp_{\delta ,q}(\rho_2)\bigcup\bigcup_{p+q=k}
Sp_{\delta ,p}(\rho_1) \times Sp_{\delta ,q,e}(\rho_2),\\\end{align*}

\begin{align*} &ii) \hbox{}\bigcup_{p+q=k}\Sigma_{\pi ,p,e}(\rho_1)\times
\Sigma_{\pi ,q}(\rho_2 )\bigcup\bigcup_{p+q=k}\Sigma_{\pi
,p}(\rho_1) \times\Sigma_{\pi ,q ,e}(\rho_2)
\subseteq\Sigma_{\pi ,k ,e}(\rho)\subseteq \\
&Sp_{\pi ,k,e} (\rho)\subseteq \bigcup_{p+q=k}Sp_{\pi ,p,e}
(\rho_1)\times Sp_{\pi ,q}(\rho_2)\bigcup\bigcup_{p+q=k}Sp_{\pi ,p}
(\rho_1) \times Sp_{\pi ,q,e}(\rho_2).\\\end{align*}

In particular, if $X_1$ and $X_2$ are Hilbert spaces, the above
inclusions are equalities.\par
\end{thm}

\begin{proof} It is a consequence of
Definitions 2.4 and 2.6, Theorem 17 of [3], and Remark 2.1
\end{proof}

\begin{thm}  Let $X_1$ and $X_2$ be two complex Banach
spaces, $L_1$ and $L_2$ two complex solvable finite dimensional Lie
algebras, and $\rho_i \colon L_i\to  L(X_i)$, $i=1$, $2$,
two representations of the Lie algebras. Suppose that $J$ is an
operator ideal between $X_2$ and $X_1$, and present it as the tensor
product of $X_1$ and $X_2{'}$ relative to $\langle
X_1,X_1{'}\rangle$ and $\langle X_2{'},X_2\rangle$. In addition,
consider the multiplication representation of $L=L_1\times
L_2^{op}$, $\tilde{\rho}\colon L\to L(J)$. Then, if dim $L_2$ = $m$,

\begin{align*} i) &\hbox{}\bigcup_{p+q=k}\Sigma_{\delta , p}(\rho_1)\times
\Sigma_{\pi , m-q}(\rho_2)\subseteq\Sigma_{\delta  ,k}
(\tilde{\rho})
\subseteq Sp_{\delta , k} (\tilde{\rho})\subseteq\\
&\bigcup_{p+q=k}Sp_{\delta , p} (\rho_1)\times Sp_{\pi ,
m-q}(\rho_2),\\\end{align*}

\begin{align*} ii)&\hbox{}\bigcup_{p+q=k}\Sigma_{\pi ,p}(\rho_1)\times
\Sigma_{\delta , m-q}(\rho_2)\subseteq \Sigma_{\pi ,k}(\tilde{\rho})
\subseteq
 Sp_{\pi , k} (\tilde{\rho})\subseteq\\
&\bigcup_{p+q=k}Sp_{\pi , p} (\rho_1)\times Sp_{\delta
,m-q}(\rho_2).\\\end{align*}

\indent In particular, if $X_1$ and $X_2$ are Hilbert spaces, the
above inclusions are equalities.\par
\end{thm}

\begin{proof}It is a consequence of
Definitions 2.2 and 2.5, Theorem 19 of [3], [17, 0.5.8], and Remark
2.1.\end{proof}

\begin{thm}  In the same conditions of Theorem 3.5,

\begin{align*} i)&\hbox{}\bigcup_{p+q=k}\Sigma_{\delta , p,e}(\rho_1)\times
\Sigma_{\pi ,m-q}(\rho_2)\bigcup\bigcup_{p+q=k}\Sigma_{\delta , p}
(\rho_1)\times \Sigma_{\pi , m-q ,e}(\rho_2)
\subseteq\\
&\Sigma_{\delta ,k ,e}(\tilde{\rho})
\subseteq Sp_{\delta , k,e} (\tilde{\rho})\subseteq \\
&\bigcup_{p+q=k}Sp_{\delta ,p ,e} (\rho_1)\times Sp_{\pi ,
m-q}(\rho_2) \bigcup\bigcup_{p+q=k} Sp_{\delta , p}(\rho_1)\times
Sp_{\pi , m-q ,e}(\rho_2),\\\end{align*}

\begin{align*} ii) &\hbox{}\bigcup_{p+q=k}\Sigma_{\pi , p, e}(\rho_1)\times
\Sigma_{\delta ,m-q}(\rho_2)\bigcup\bigcup_{p+q=k} \Sigma_{\pi
,p}(\rho_1)\times \Sigma_{\delta ,m-q ,e}(\rho_2)
\subseteq\\
&\Sigma_{\pi , k ,e}(\tilde{\rho})
\subseteq Sp_{\delta , k ,e} (\tilde{\rho})\subseteq\\
& \bigcup_{p+q=k}Sp_{\pi ,p ,e} (\rho_1)\times Sp_{\delta ,
m-q}(\rho_2)\bigcup \bigcup_{p+q=k}Sp_{\pi , p}(\rho_1)\times
sp_{\delta ,m-q ,e}(\rho_2) .\\\end{align*}

\indent In particular, if $X_1$ and $X_2$ are Hilbert spaces, the
above inclusions are equalities.\par
\end{thm}

\begin{proof}
\indent It is a consequence of Definition 2.4 and 2.6, Theorem 20 of
[3], [17, 0.5.8], and Remark 2.1.\end{proof}

\indent A final remark. It is worth noticing that as in the previous
results, the analogues of Theorems 23 and 24 of [3] can be also
proved for the joint spectra of Definitions 2.2 and 2.4 - 2.6.\par

\vskip.5cm
\noindent Enrico Boasso\par
\noindent E-mail address: enrico\_odisseo@yahoo.it

\end{document}